\documentclass[12pt,a4paper,reqno]{amsart}
\usepackage{amsmath,amssymb,amsthm,wasysym,calc,verbatim,enumitem,tikz,url,hyperref,mathrsfs,bbm,cite,fullpage,bm}
\usetikzlibrary{shapes.misc,calc,intersections,patterns,decorations.pathreplacing}
\hypersetup{colorlinks=true,citecolor=blue,linktoc=page}

\usepackage[utf8]{inputenc}
\usepackage[english]{babel}
\usepackage{stmaryrd}

\usepackage[abbrev,msc-links,backrefs]{amsrefs}
\usepackage{doi}

\renewcommand{\PrintDOI}[1]{\doi{#1}}

\AtBeginDocument{\def\MR#1{}}

\usepackage{tikz}
\usetikzlibrary{decorations.markings}
\usetikzlibrary{calc,positioning,decorations.pathmorphing,decorations.pathreplacing}

\usepackage{fullpage}
\usepackage{setspace}
\usepackage{datetime}
\usepackage[margin=1cm]{caption}

\usepackage{enumitem}

\let\polishlcross=\l
\def\l{\ifmmode\ell\else\polishlcross\fi}

\makeatletter
\def\moverlay{\mathpalette\mov@rlay}
\def\mov@rlay#1#2{\leavevmode\vtop{    \baselineskip\z@skip\lineskiplimit-\maxdimen%
    \ialign{\hfil$\m@th#1##$\hfil\cr#2\crcr}}}
\newcommand{\charfusion}[3][\mathord]{
    #1{\ifx#1\mathop\vphantom{#2}\fi
        \mathpalette\mov@rlay{#2\cr#3}
      }
    \ifx#1\mathop\expandafter\displaylimits\fi}
\makeatother

\newcommand{\oldqed}{}
\def\endofFact{\hfill\scalebox{.6}{$\Box$}}
\newenvironment{claimproof}[1][Proof]{
  \renewcommand{\oldqed}{\qedsymbol}
  \renewcommand{\qedsymbol}{\endofFact}
  \begin{proof}[#1]
}{
  \end{proof}
  \renewcommand{\qedsymbol}{\oldqed}
}

\DeclareFontFamily{U}  {MnSymbolC}{}
\DeclareSymbolFont{MnSyC}         {U}  {MnSymbolC}{m}{n}
\DeclareFontShape{U}{MnSymbolC}{m}{n}{%
    <-6>  MnSymbolC5
   <6-7>  MnSymbolC6
   <7-8>  MnSymbolC7
   <8-9>  MnSymbolC8
   <9-10> MnSymbolC9
  <10-12> MnSymbolC10
  <12->   MnSymbolC12}{}
\DeclareMathSymbol{\powerset}{\mathord}{MnSyC}{180}

\makeatletter
\def\namedlabel#1#2{\begingroup
    #2%
    \def\@currentlabel{#2}%
    \phantomsection\label{#1}\endgroup
}
\makeatother

\setlist[itemize]{leftmargin=1cm}
\setlist[enumerate]{leftmargin=1cm}

\everydisplay{\thickmuskip=7mu}
\def\circ{K^\circlearrowright}
\def\circcycle{C^\circlearrowright}

\newtheorem{theorem}{Theorem}[section]
\newtheorem{problem}{Problem}
\newtheorem*{definition}{Definition}
\newtheorem*{theoremnull}{Theorem}
\newtheorem{conjecture}{Conjecture}

\newtheorem{proposition}[theorem]{Proposition}
\newtheorem{lemma}[theorem]{Lemma}
\newtheorem{corollary}[theorem]{Corollary}

\newtheorem*{claim}{Claim}

\renewcommand{\O}[1]{\mathcal{D}(#1,\circ_3)}
\renewcommand{\o}[1]{D(#1,\circ_3)}
\newcommand{\ext}{{\rm{ext}}}

\newcommand{\ex}{\operatorname{ex}}
\newcommand{\bound}{{\lfloor n^2/4\rfloor}}

\usepackage{./figstyle}

\usepackage{subcaption}
\usepackage{graphicx}

\title{Counting graph orientations\\with no directed triangles}

\begin{document}
\onehalfspace%
\footskip=28pt

\author[P.~Araújo]{Pedro Araújo}

\author[F.~Botler]{Fábio Botler}

\author[G.~O.~Mota]{Guilherme Oliveira Mota}

\address{IMPA, Estrada Dona Castorina 110, Jardim Bot\^anico, Rio de Janeiro, RJ, Brazil}
\email{pedroc@impa.br}

\address{Programa de Engenharia de Sistemas e Computação, Universidade Federal do Rio de Janeiro, Rio de Janeiro, Brazil}
\email{fbotler@cos.ufrj.br}

\address{Centro de Matem\'atica, Computa\c c\~ao e Cogni\c c\~ao \\Universidade Federal do ABC, Santo Andr\'e, Brazil}
\email{g.mota@ufabc.edu.br}

\thanks{\tiny
P.~Araújo was partially supported by CNPq;
F. Botler was supported by CNPq (423395/2018-1), and by FAPERJ (211.305/2019);
G.~O.~Mota was partially supported by CNPq (304733/2017-2, 428385/2018-4) and FAPESP (2018/04876-1, 2019/13364-7).
The research that led to this paper started in WoPOCA 2019, which was financed by FAPESP (2015/11937-9) and CNPq (425340/2016-3, 423833/2018-9).
This study was financed in part by the Coordena\c{c}\~ao de Aperfei\c{c}oamento de Pessoal de N\'ivel Superior, Brasil (CAPES), Finance Code~001.
FAPERJ is the Rio de Janeiro Research Foundation. 
FAPESP is the S\~ao Paulo Research Foundation.  
CNPq is the National Council for Scientific and Technological Development of Brazil.}

\begin{abstract}
Alon and Yuster proved that the number of orientations of any $n$-vertex graph in which every $K_3$ is transitively oriented is at most $2^{\lfloor n^2/4\rfloor}$ for $n \geq 10^4$ and conjectured that the precise lower bound on $n$ should be $n \geq 8$. We confirm their conjecture and, additionally, characterize the extremal families by showing that the balanced complete bipartite graph with $n$ vertices is the only $n$-vertex graph for which there are exactly $2^{\lfloor n^2/4\rfloor}$ such orientations.
\end{abstract}

\maketitle

\section{Introduction}

Given a graph $G$ and an oriented graph $\vec H$, we say that $\vec G$ is an \emph{$\vec H$-free orientation} of $G$ if $\vec G$ contains no copy of $\vec H$.
We denote by $\mathcal{D}(G,\vec H)$ the family of $\vec H$-free orientations of $G$ and we write $D(G,\vec H) = |\mathcal{D}(G,\vec H)|$.
In 1974, Erd\H{o}s~\cite{Er74} posed the problem of determining the maximum number of $\vec{H}$-free orientations of~$G$, for every $n$-vertex graph $G$.
Formally, we define
$
D(n,\vec H) = \max\{\mathcal{D}(G,\vec H)\colon G\text{ is an $n$-vertex graph}\}.
$

Since every orientation of an $H$-free graph does not contain any orientation $\vec H$ of $H$, it is fairly straightforward to see that $D(n,\vec H)\geq 2^{\ex(n,H)}$,
where $\ex(n,H)$ is the maximum number of edges in an $H$-free graph on $n$ vertices.
For a tournament $\vec T_k$ on $k$ vertices, Alon and Yuster~\cite{AlYu06} proved that~$D(n,\vec T_k) = 2^{\ex(n,K_k)}$ for $n\geq n_0$ with a very large $n_0$, as they use the Regularity Lemma~\cite{Sz75}. 
For tournaments with three vertices, they avoid using the regularity lemma to prove that $D(n,T_3) = 2^{\lfloor n^2/4\rfloor}$ for $n\geq n_0$, where $n_0$ is slightly less than~$10000$.
Furthermore, for the strongly connected triangle, denoted by $\circ_3$, using a computer program they verified that $D(8,\circ_3)=2^{16}$ and $D(n,\circ_3)=n!$ for $n\leq 7$.
In view of this, Alon and Yuster posed the following conjecture.
\begin{conjecture}[Alon and Yuster~\cite{AlYu06}]
\label{conj:AY}
For $n\geq 1$, we have
	$
	D(n,\circ_3) = \max\{2^{\lfloor n^2/4\rfloor},n!\}.
	$
\end{conjecture}

Using a simple computer program, we checked that \(K_{4,4}\) is the only $8$-vertex graph that maximizes $D(8,\circ_3)$.
This fact together with the verification made by Alon and Yuster for graphs with at most seven vertices implies the following proposition.
\begin{proposition}\label{prop:Alon-small-cases}
$D(8,\circ_3)=2^{16}$ and among all graphs with $8$ vertices, $D(G,\circ_3)=2^{16}$ if and only if $G \simeq K_{4,4}$. Furthermore, $D(n,\circ_3)=n!$ for $1\leq n\leq 7$.
\end{proposition}

In this paper we prove the following result that confirms Conjecture~\ref{conj:AY} and states that the balanced complete bipartite graph is the only $n$-vertex graph for which there are exactly $2^{\lfloor n^2/4\rfloor}$ orientations with no copy of $\circ_3$.

\begin{theorem}\label{thm:main}
	For $n\geq 8$, we have
	$
	D(n,\circ_3) = 2^{\lfloor n^2/4\rfloor}.
	$
	Furthermore, among all graphs~$G$ with $n$ vertices, $D(G,\circ_3)=2^{\lfloor n^2/4\rfloor}$ if and only if $G \simeq K_{\lfloor n/2\rfloor, \lceil n/2\rceil}$.
\end{theorem}

\smallskip\noindent
\textbf{Overview of the paper.}
Our proof is divided into two parts. 
Proposition~\ref{prop:small-cases} deals with graphs with at most \(13\) vertices,
and its proof is given in the appendix (Section~\ref{sec:appendix});
and Theorem~\ref{thm:main} deals with general graphs (Section~\ref{sec:main}).
The proofs of these results are somehow similar and consist of an analysis of the size of a maximum clique of the given graph.
In each step, we partition the vertices of a graph~\(G\) into a few parts and, 
using the results presented in Section~\ref{sec:comp},
explore the orientations of the edges between these parts
that lead to $\circ_3$-free orientations of $G$.
Our proof is then reduced to solving a few equations which, in the case of the proof of Proposition~\ref{prop:small-cases}, can be checked by straightforward computer programs. 
In Section~\ref{sec:conc} we present some open problems.
The reader is referred to \cites{Bo78,Di10} for standard terminology on graphs.

\section{Extensions of $\circ_3$-free orientations}
\label{sec:comp}

In this section we provide several bounds on the number of ways one can extend a $\circ_3$-free orientation of a subgraph of a graph $G$ to a $\circ_3$-free orientation of $G$.

Given subgraphs $G_1$ and $G_2$ of $G$, we write $G_1\cup G_2$ for the subgraph of $G$ with vertex set $V(G_1)\cup V(G_2)$ and edge set $E(G_1)\cup E(G_2)$.
Let $\vec G_1$ and $\vec G_2$ be orientations, respectively, of $G_1$ and $G_2$ with the property that any edge of $E(G_1)\cap E(G_2)$ gets the same orientation in $\vec G_1$ and $\vec G_2$.
We denote by $\vec G_1\cup \vec G_2$ the orientation of $G_1\cup G_2$ following the orientations $\vec G_1$ and $\vec G_2$.
 
Let $G$ be a graph and $S\subseteq E(G)$.
For simplicity, we say that an orientation of the subgraph ${G[S]}$ of $G$ induced by the set of edges $S$ is an orientation of $S$.
The next definition is a central concept of this paper.

\begin{definition}[Compatible orientations]
Given a graph $G$, disjoint sets $S$, $T\subseteq E(G)$ and orientations $\vec S$ of $S$ and $\vec T$ of $T$, we say that $\vec S$ and $\vec T$ are \emph{compatible} if $\vec{S}\cup\vec{T}$ is $\circ_3$-free.
\end{definition}

Given a graph $G$ and disjoint sets $A,B\subseteq V(G)$, denote by $E_G(A,B)$ the set of edges of $G$ between $A$ and $B$ and by $G[A,B]$ the spanning subgraph of $G$ induced by $E_G(A,B)$.
It is useful to have an upper bound on number of $\circ_3$-free orientations of $E_G(A,B)$ that are compatible with a fixed orientation of $G[A]\cup G[B]$. This quantity is precisely the maximum number of ways one can \emph{extend} a $\circ_3$-free orientation of $G[A]\cup G[B]$ to a $\circ_3$-free orientation of $G[A\cup B]$.

\begin{definition}\label{def:maxcomp}
Given a graph $G$ and disjoint sets $A$, $B\subseteq V(G)$, let $T=G[A]\cup G[B]$. We define $\ext_G(A,B)$ as follows:
$$
\ext_G(A,B) = \max_{\vec T \in \O{T}} |\{\vec S\in \O{G[A,B]}\colon \text{$\vec S$ and $\vec T$ are compatible}\}|.
$$
\end{definition}
For simplicity, when $A=\{u\}$, we write $\ext_G(u,B)$ instead of $\ext_G(\{u\},B)$.
In the rest of this section we give upper bounds for $\ext_G(A,B)$ for specific graphs $G$ and subgraphs $G[A]$ and $G[B]$. 
If $A$ induces a complete graph with $k$ vertices, then we remark that any $\circ_3$-free orientation $\vec S$ of $G[A]$ is a transitive orientation, which thus induces a unique ordering $(v_1,\ldots, v_k)$ of the vertices of $A$, called \emph{the transitive ordering of $\vec S$}, such that every edge $\{v_i,v_j\}$ ($1\leq i<j\leq k$) is oriented from $v_i$ to $v_j$ in $\vec S$. 

Given a graph $G$, a vertex $v\in V(G)$ and a clique $W\subseteq V(G)\setminus\{v\}$, we denote by $d_G(v,W)$ the number of neighbors of $v$ in $W$.
Consider a $\circ_3$-free orientation $\vec W$ of $G[W]$ and note that if we have a transitive ordering $(w_1,\ldots,w_k)$ of $\vec W$, then there are exactly $d_G(v,W)+1$ ways to extend this ordering to a transitive ordering of $v\cup W$, as it depends only on the position in which we place $v$ in $(w_1,\ldots,w_k)$ with respect to its neighbors in~$W$ (there are $d_G(v,W)+1$ such positions). 
We summarize this discussion in the following proposition.

\begin{proposition}\label{lemma:compatible-vertex-clique}
	Given a graph $G$, $v\in V(G)$ and $W\subseteq V(G)\setminus\{v\}$.
	If $G[W]$ is a complete graph, then $\ext_G(v,W) = d_G(v,W)+1$.
\end{proposition}

In the next two results, we give an upper bound for $\ext_G(A,B)$ when $A$ induces a complete graph and $B=\{u,v\}$ is an edge.
We denote by $d_A(x)$ the neighborhood of $x$ in $A$ and \(d_A(x,y)\) denotes the number of common neighbors of \(x\) and \(y\) in \(A\).

\begin{lemma}\label{lemma:compatible-edge-Kr}
	Let $r\geq 3$ be an integer and let $G$ be a graph.
	If $A,B\subseteq V(G)$ induce disjoint cliques with $|A|=r$ and $B=\{u,v\}$ such that $d_A(x,y)\neq 0$, then
	\[
		\ext_G(A,B)\leq (d_{A}(u)+1)(d_{A}(v)+1)-\dbinom{d_{A}(u,v)+1}{2}.
	\]
\end{lemma}

\begin{proof}
Let $\vec A$ and $\vec{B}$ be arbitrary $\circ_3$-free orientations of $G[A]$ and $G[B]$ respectively.
Suppose without loss of generality that $\vec B$ assigns the orientation of $\{u,v\}$ from $u$ to~$v$ and consider the transitive ordering of $\vec A$.
We estimate in how many ways one can include $u$ and $v$ in the ordering $(v_1,\ldots, v_r)$ while keeping it transitive.
Since $\{u\}\cup N_A(u)$ and $\{v\}\cup N_A(v)$ are cliques, by Proposition~\ref{lemma:compatible-vertex-clique} we have $\ext_G(u,A)\leq d_{A}(u)+1$ and $\ext_G(v,A)\leq d_{A}(v)+1$, which gives at most $(d_A(u)+1)(d_A(v)+1)$ positions to put the vertices $u$ and $v$ in the transitive ordering of~$\vec A$.
Note that there are $\binom{d_{A}(u,v)+1}{2}$ ways to place $\{u,v\}$ in the transitive ordering of $\vec A$ such that $u$ appears after $v$ and they have a common neighbor between them.
But each such ordering induces a $\circ_3$. This finishes the proof.
\end{proof}

The following corollary bounds the number of extensions $\ext_G(A,B)$ when $A$ is a maximum clique of \(G\) and $B=\{u,v\}$ is an edge.

\begin{corollary}\label{cor:compatible-edge-Kr}
	Let $r\geq 2$ be an integer and let $G$ be a $K_{r+1}$-free graph.
	If $A,B\subseteq V(G)$ are disjoint cliques with $|A|=r$ and $B=\{x,y\}$, then
	$$
	\ext_G(A,B) \leq r^2 - {r-1\choose 2}.
	$$
\end{corollary}

\begin{proof}
	Let \(d_x = d_G(x,A)\) and \(d_y = d_G(y,A)\), and put $d=d_x+d_y$. 
	If $d\leq r$, then by applying Proposition~\ref{lemma:compatible-vertex-clique} twice, with $x$ and $y$, we have
	\begin{equation*}
		\ext_G(A,B) \leq (d_x+1)(d_y+1) \leq \dfrac{d^2}{4} + d+1 \leq \frac{r^2}{4}+r+1 \leq r^2 - {r-1\choose 2}.
	\end{equation*}
	Therefore, we assume that $d>r$.
	Note that since $G$ is $K_{r+1}$-free, we have $d_x,d_y\leq r-1$.
	Applying Lemma~\ref{lemma:compatible-edge-Kr} and using the fact that, for $d>r$, we have $d_A(x,y) \geq  d-r$, we obtain
	\begin{equation}\label{eq:qualquer}
		\ext_G(A,B)\leq (d_x+1)(d_y+1) - {d-r+1\choose 2} \leq	\dfrac{d^2}{4}+d+1-\dbinom{d-r+1}{2}.
	\end{equation}
One can check that the right-hand side of~\eqref{eq:qualquer} is a polynomial on $d$ of degree $2$ with negative leading coefficient and it is a growing function in the interval $(-\infty,2r+1)$. Since  $d\leq 2(r-1)$, we have 
	\[\ext_G(A,B)\leq (r-1)^2 + 2(r-1)+1 -\dbinom{r-1}{2} = r^2-\dbinom{r-1}{2}.\]
\end{proof}

Given a graph \(G\), an edge \(e\), and an orientation \(\vec{S}\) of \(E(G)\setminus\{e\}\),
we say that the orientation of \(e\) is \emph{forced} if there is only one orientation 
of \(e\) compatible with \(\vec{S}\).
In the next two lemmas we provide bounds for the number of $\circ_3$-free orientations of $K_4$-free graphs.

\begin{lemma}\label{lemma:compatible-K2-K2}
	Let $G$ be a $K_4$-free graph and let $A$, $B\subseteq V(G)$ be disjoint cliques of size~$2$.
	Then $\ext_G(A,B) \leq 5$.
\end{lemma}

\begin{proof}
	First, note that if $e_G(A,B)\leq 2$, then the trivial bound $\ext_G(A,B) \leq 2^{e(A,B)}$ 
	implies $\ext_G(A,B) \leq 4$.
	Also, since $G$ is $K_4$-free, we have $e_G(A,B)\leq 3$.
	Thus, we may assume that $e_G(A,B)= 3$, i.e., $G[A\cup B]$ is a $K_4^-$.
	Let $A=\{u_1,u_2\}$ and $B=\{v_1,v_2\}$ so that $u_2v_2$ is not an edge and 	consider an arbitrary orientation of the edges $\{u_1,u_2\}$ and $\{v_1,v_2\}$.

	If the oriented edges are $u_1u_2$ and $v_1v_2$ (or, by symmetry, $u_2u_1$ and $v_2v_1$), then for the two possible orientations of $\{u_1,v_1\}$, the orientation of one of the two remaining edges in $E_G(A,B)$ is forced. 
Thus, since there is only one edge left to orient in $E_G(A,B)$, which can be done in two ways, we have $\ext_G(A,B)\leq 4$.
	
	It remains to consider the case where the oriented edges are $u_1u_2$ and $v_2v_1$ (or, by symmetry, $u_2u_1$ and $v_1v_2$).
	If $\{u_1v_1\}$ is oriented from $u_1$ to $v_1$, then the orientation of the two remaining edges in $E_G(A,B)$ are forced, which gives us one $\circ_3$-free orientation.
	On the other hand, if $\{u_1v_1\}$ is oriented from $v_1$ to $u_1$, then one can orient the both remaining edges in $E(A,B)$ in two ways, which in total gives that $\ext_G(A,B)\leq 5$.
\end{proof}

\begin{lemma}\label{lemma:compatible-vertex-K4-}
	Let $G$ be a $K_4$-free graph and let $u\in V(G)$ and $B\subseteq V(G)$ with $|B|=4$.
	If $G[B]$ induces a copy of $K_4^-$, then $\ext_G(u,B) \leq 5$.
\end{lemma}

\begin{proof}
	Consider an arbitrary orientation of the edges of $G[B]$.
	We may assume that $d_B(u)\geq 3$, as otherwise we have $\ext_G(u,B)\leq 4$.
	Since $G$ is $K_4$-free and $G[B]$ induces a copy of $K_4^{-}$, the vertex $u$ must have exactly three neighbors in $B$, which span an induced path $v_1v_2v_3$.
	By symmetry, we assume that $\{v_1,v_2\}$ is oriented from $v_1$ to $v_2$. 
	If we orient $u v_1$ from $u$ to $v_1$, then the orientation of $\{u,v_2\}$ is forced, which leaves two possible orientations for the edge $\{u,v_3\}$.
	On the other hand, if we orient $u v_1$ from $u$ to $v_1$, we just apply Proposition~\ref{lemma:compatible-vertex-clique} to conclude that $\ext_G(u, \{v_2,v_3\})\leq 3$. 
	Combining the possible orientations, we obtain $\ext_G(u,B)\leq 5$.	
\end{proof}

We now provide an upper bound for $\ext_G(A,B)$ (see Lemma~\ref{lemma:compatible-K3-K3} below) in a specific configuration of a $K_4$-free graph $G$,  and subsets of vertices $A$ and $B$, which is proved using the following proposition.

\begin{proposition}\label{lemma:compatible-P5-square}
	Let $P$ be a path $abcde$,
	and let $T = \{ac,bd,ce\}$.
	Given an orientation $\vec T$ of $T$,
	there are at most eight orientations of $E(P)$ compatible with $\vec T$.
	Moreover, if the edges $\{a,c\}$ and $\{b,d\}$ are oriented, respectively, towards $a$ and $d$,
	then there are at most~$7$ such orientations.
\end{proposition}

\begin{proof}
	By Proposition~\ref{lemma:compatible-vertex-clique}, there are three orientations of $T_1 = G[\{a,b,c\}]$ (resp. $T_2=G[\{c,d,e\}]$) compatible with $\vec T$,
	and hence there are at most nine orientations of $E(P)$ compatible with $\vec T$.
	In these orientations, each direction of $\{b,c\}$ and $\{c,d\}$ appears at least once.
	If $\{b,d\}$ is oriented towards $d$ (resp. towards $b$),
	then the orientations in which $\{b,c\}$ and $\{c,d\}$ are oriented, respectively, towards $b$ and $c$ (resp. $c$ and $d$) are not compatible with $\vec T$.
	Therefore, there are at most eight orientations of $E(P)$ compatible with $\vec T$.
	Now, suppose that $\{a,c\}$ and $\{b,d\}$ are oriented towards $a$ and $d$.
	If we orient $\{b,c\}$ towards $c$ (resp. $b$), then $\{a,b\}$ must be oriented towards $a$ (resp. $\{c,d\}$ must be oriented towards $d$),
	and there are three orientations of $E(T_2)$ (resp. four orientations of $\{\{a,b\},\{d,e\}\}$) from which we can complete a compatible orientation of $E(P)$.
	Therefore, there are at most seven orientations of \(E(P)\).
\end{proof}

\begin{lemma}\label{lemma:compatible-K3-K3}
	Let $G$ be a $K_4$-free graph and let $A$, $B\subseteq V(G)$ be disjoint cliques of size~$3$.
	Then $\ext_G(A,B) \leq 15$.
\end{lemma}

\begin{proof}
	Let $A = \{x_1,x_2,x_3\}$ and $B=\{y_1,y_2,y_3\}$.
	Since $G$ is $K_4$-free, $y_i$ cannot be adjacent to every vertex of $A$, for $i=1,2,3$.
	This implies that $d_A(y_i)\leq 2$, for $i=1,2,3$.
	Analogously, we have $d_B(x_i)\leq 2$, for $i=1,2,3$.
	Thus the set $E$ of edges in $G$ joining $A$ and $B$ induces a set of paths and cycles.
	Since $G$ is $K_4$-free, $E$ does contain a cycle of length $4$.
	If $|E|\leq 3$, then $\ext_G(A,B) \leq 2^{|E|}\leq 8$, as desired.
	If $|E| = 4$, then some vertex, say $x_1\in A$, is incident to two edges of $E$, say $\{x_1,y_1\}$ and $\{x_1,y_2\}$,
	which implies that $\ext_G(x_1,B)\leq 3$,
	and hence $\ext_G(A,B)\leq \ext_G(x_1,B)\cdot 2^{|E\setminus\{\{x_1,y_1\},\{x_1,y_2\}\}|} \leq 12$, as desired.
	If $|E| = 5$, then $|E|$ induces a path of length $5$, say $x_1y_1x_2y_2x_3y_3$.
	In this case, note that $\{x_1,x_2\}$ and $\{y_1,y_2\}$ are disjoint cliques of size $2$,
	and hence, by Lemma~\ref{lemma:compatible-K2-K2}, we have $\ext_G(\{x_1,x_2\},\{y_1,y_2\})\leq 5$.
	Since each edge in $E$ either joins $\{x_1,x_2\}$ to $\{y_1,y_2\}$, or is adjacent to $x_3$,
	we have $\ext_G(A,B)\leq \ext_G(\{x_1,x_2\},\{y_1,y_2\})\cdot\ext_G(x_3,B) \leq 15$, as desired.

	Thus, we may assume $|E|=6$, and hence $E$ induces the cycle $x_1y_1x_2y_2x_3y_3x_1$.
	By symmetry, we may assume that $\{x_1,x_2\}$, $\{x_1,x_3\}$ are both oriented towards $x_1$, and $\{y_1,y_3\}$ is oriented towards $y_1$.
	Suppose $\{y_2,y_3\}$ is oriented towards $y_2$.
	If we orient $\{x_1,y_1\}$ towards $y_1$, then $x_2y_1$ must be oriented towards $y_1$,
	and, since $\{x_1,x_3\}$ and $\{y_2,y_3\}$ are oriented towards $x_1$ and $y_2$,
	by Proposition~\ref{lemma:compatible-P5-square}, 
	there are $7$ compatible orientations of the edges in the path $x_2y_2x_3y_3x_1$ 
	(see Figure~\ref{fig:k3k3-1}).
	If we orient $\{x_1,y_1\}$ towards $x_1$, then $\{y_3,x_1\}$ must be oriented towards $x_1$,
	and by Proposition~\ref{lemma:compatible-P5-square}, 
	there are $8$ compatible orientations of the edges in the path $y_1x_2y_2x_3y_3$
	(see Figure~\ref{fig:k3k3-2}).
	Thus, there are $15$ compatible orientations of $E$, as desired.
	Thus, we may assume that $\{y_2,y_3\}$ is oriented towards $y_3$,
	and hence $\{y_1,y_2\}$ must be oriented towards $y_1$.
	If we orient $\{x_1,y_1\}$ towards $y_1$, then $\{x_2,y_1\}$ must be oriented towards $y_1$,
	and by Proposition~\ref{lemma:compatible-P5-square}, 
	there are $8$ compatible orientations of the edges in the path $x_2y_2x_3y_3x_1$
	(see Figure~\ref{fig:k3k3-3}).
	If we orient $\{x_1,y_1\}$ towards $x_1$, then $\{y_3,x_1\}$ must be oriented towards $x_1$,
	and, since $\{x_1,x_3\}$ and $\{x_1,x_2\}$ are oriented towards $x_1$, regardless of the orientation of $\{x_2,x_3\}$,
	by Proposition~\ref{lemma:compatible-P5-square}, 
	there are $7$ compatible orientations of the edges in the path $y_1x_2y_2x_3y_3$
	(see Figure~\ref{fig:k3k3-4})
	Thus, there are $15$ compatible orientations of $E$, as desired.
\end{proof}

\begin{figure}
	\centering
    \begin{subfigure}{.23\textwidth}
		\begin{tikzpicture}[scale = 0.4]

	\foreach \i in {1,2,3}
	{
		\node (x\i) [black vertex] at (\i*120-30:3) {};
		\node ()				   at (\i*120-30:3.7) {$x_\i$};
		\node (y\i) [black vertex] at (\i*120+30:1) {};
		\node ()				   at (\i*120+30:1.7) {$y_\i$};
	}
	
	\draw[very thick,<-]				(x1) to [bend right] (x2);
	\draw[very thick]					(x2) to [bend right] (x3);
	\draw[very thick,->]				(x3) to [bend right] (x1);
	\draw[very thick]					(y1) -- (y2);
	\draw[very thick,<-,color=yellow!50!green]	(y2) -- (y3);
	\draw[very thick,->]				(y3) -- (y1);
	\draw[very thick,->,color=red]		(x1) -- (y1);
	\draw[very thick,<-,color=blue]		(y1) -- (x2);
	\draw[very thick]					(x2) -- (y2);
	\draw[very thick]					(y2) -- (x3);
	\draw[very thick]					(x3) -- (y3);
	\draw[very thick]					(y3) -- (x1);
\end{tikzpicture}
		\caption{}\label{fig:k3k3-1}
    \end{subfigure}
    \begin{subfigure}{.23\textwidth}
		\begin{tikzpicture}[scale = 0.4]

	\foreach \i in {1,2,3}
	{
		\node (x\i) [black vertex] at (\i*120-30:3) {};
		\node ()				   at (\i*120-30:3.7) {$x_\i$};
		\node (y\i) [black vertex] at (\i*120+30:1) {};
		\node ()				   at (\i*120+30:1.7) {$y_\i$};
	}
	
	\draw[very thick,<-]				(x1) to [bend right] (x2);
	\draw[very thick]					(x2) to [bend right] (x3);
	\draw[very thick,->]				(x3) to [bend right] (x1);
	\draw[very thick]					(y1) -- (y2);
	\draw[very thick,<-,color=yellow!50!green]	(y2) -- (y3);
	\draw[very thick,->]				(y3) -- (y1);
	\draw[very thick,<-,color=red]		(x1) -- (y1);
	\draw[very thick]					(y1) -- (x2);
	\draw[very thick]					(x2) -- (y2);
	\draw[very thick]					(y2) -- (x3);
	\draw[very thick]					(x3) -- (y3);
	\draw[very thick,->,color=blue]					(y3) -- (x1);
\end{tikzpicture}
		\caption{}\label{fig:k3k3-2}
    \end{subfigure}
    \begin{subfigure}{.23\textwidth}
		\begin{tikzpicture}[scale = 0.4]

	\foreach \i in {1,2,3}
	{
		\node (x\i) [black vertex] at (\i*120-30:3) {};
		\node ()				   at (\i*120-30:3.7) {$x_\i$};
		\node (y\i) [black vertex] at (\i*120+30:1) {};
		\node ()				   at (\i*120+30:1.7) {$y_\i$};
	}
	
	\draw[very thick,<-]				(x1) to [bend right] (x2);
	\draw[very thick]					(x2) to [bend right] (x3);
	\draw[very thick,->]				(x3) to [bend right] (x1);
	\draw[very thick,<-]					(y1) -- (y2);
	\draw[very thick,->,color=yellow!50!green]	(y2) -- (y3);
	\draw[very thick,->]				(y3) -- (y1);
	\draw[very thick,->,color=red]		(x1) -- (y1);
	\draw[very thick,<-,color=blue]		(y1) -- (x2);
	\draw[very thick]					(x2) -- (y2);
	\draw[very thick]					(y2) -- (x3);
	\draw[very thick]					(x3) -- (y3);
	\draw[very thick]					(y3) -- (x1);
\end{tikzpicture}
		\caption{}\label{fig:k3k3-3}
    \end{subfigure}
    \begin{subfigure}{.23\textwidth}
		\begin{tikzpicture}[scale = 0.4]

	\foreach \i in {1,2,3}
	{
		\node (x\i) [black vertex] at (\i*120-30:3) {};
		\node ()				   at (\i*120-30:3.7) {$x_\i$};
		\node (y\i) [black vertex] at (\i*120+30:1) {};
		\node ()				   at (\i*120+30:1.7) {$y_\i$};
	}
	
	\draw[very thick,<-]				(x1) to [bend right] (x2);
	\draw[very thick]					(x2) to [bend right] (x3);
	\draw[very thick,->]				(x3) to [bend right] (x1);
	\draw[very thick,<-]					(y1) -- (y2);
	\draw[very thick,->,color=yellow!50!green]	(y2) -- (y3);
	\draw[very thick,->]				(y3) -- (y1);
	\draw[very thick,<-,color=red]		(x1) -- (y1);
	\draw[very thick]					(y1) -- (x2);
	\draw[very thick]					(x2) -- (y2);
	\draw[very thick]					(y2) -- (x3);
	\draw[very thick]					(x3) -- (y3);
	\draw[very thick,->,color=blue]					(y3) -- (x1);
\end{tikzpicture}
		\caption{}\label{fig:k3k3-4}
    \end{subfigure}
    \caption{\small Compatible orientations between two cliques of size $3$ in a $K_4$-free graph.}
    \label{fig:k3k3}
\end{figure}
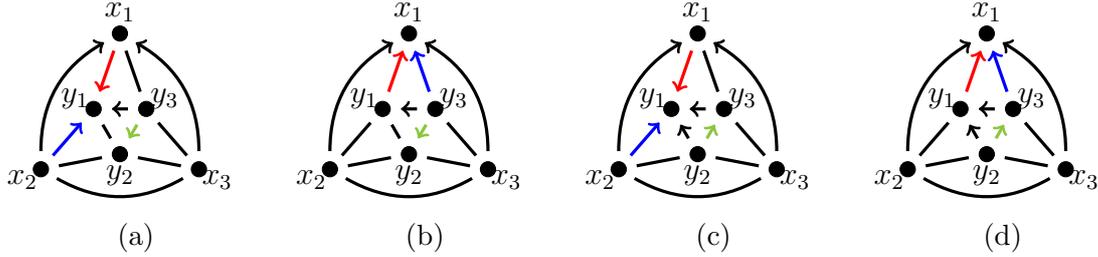

\section{Proof of the main theorem}
\label{sec:main}

In this section we prove our main result, Theorem~\ref{thm:main}.
In order to bound the number of $\circ_3$-free orientations of a graph $G$, we decompose it into disjoint cliques of different sizes and we use the results of Section~\ref{sec:comp} to bound the number of extensions of $\circ_3$-free orientations between those cliques. Before moving to the proof of the main theorem though, we need bounds on the number of $\circ_3$-free orientations of some small graphs. The first one concerns the complete tripartite graph $K_{1,\ell,\ell}$.

\begin{proposition}\label{prop:num_of_tfo-vertex-biclique}
	For any positive integer $\ell$, we have
	$$
	\o{K_{1,\ell,\ell}} = \sum_{i=0}^\ell\sum_{j=0}^\ell {\ell\choose i}{\ell\choose j} 2^{(\ell-i)j + (\ell-j)i} .
	$$
\end{proposition}

\begin{proof}
	Let $K_{1,\ell,\ell}$ be a complete tripartite graph with vertex partition $\{v\}\cup A\cup B$. For $1\leq i,j\leq \ell$ there are $\binom{\ell}{i}\binom{\ell}{j}$ orientations of the edges incident to $v$ with exactly $i$ out-neighbors of $v$ in $A$, and $j$ in-neighbors of $v$ in $B$, sets which we denote by $A^+$ and $B^-$ respectively. For each of those orientations, the edges between $A^+$ and $B^-$ and between $A\setminus A^+$ and $B\setminus B^-$ are forced in any $\circ_3$-free orientation. Since any of the other $(\ell-i)j + (\ell-j)i$ edges can be oriented in two ways, we sum over $i$ and $j$ to get
	\[
	\o{K_{1,\ell,\ell}} = \sum_{i=0}^\ell\sum_{j=0}^\ell {\ell\choose i}{\ell\choose j} 2^{(\ell-i)j + (\ell-j)i}.
	\]
\end{proof}

In the rest of the paper, we count the number of $\circ_3$-free orientations of a graph by decomposing its vertex set and we often use the following inequality without explicit reference.
For a partition of the vertices of a graph $G$ into sets $A$ and $B$ we have, from the definition of $\ext_G(A,B)$, that
	\begin{equation*}
	\o{G}	\leq \o{G[A]}\cdot \ext_G(A,B) \cdot\o{G[B]},
	\end{equation*}
When \(A\) is a clique, 
we define $m_{A,B}=\max\{|N(v,B)|+1: v\in A\}$ and use the bound
 \[\ext_G(A,B) \leq (m_{A,B})^{|A|}.\]
 In the proof of Theorem \ref{thm:main}, we first show  that $\o{G}<2^{\lfloor n^2/4\rfloor}$ for every graph containing a $K_4$. For $K_4$-free graphs we may use Lemma~\ref{lemma:compatible-K3-K3} to bound the number of extensions between two triangles.
 But when considering graphs with no two disjoint triangles, we need the following result.
 
\begin{lemma}\label{lemma:claim}
Let $H$ be a $K_4$-free graph with $7$ vertices that contains a triangle $T$, a matching $\{e_1,e_2\}$ such that $e_1$ and $e_2$ are not incident to the vertices of $T$, and that does not contains two vertex-disjoint triangles.
Then, $\o{H}< 2^{12}$.
\end{lemma}

\begin{proof}
Let $H$ be as in the statement. Recall that $\o{T} = 6$ and $\o{e_1} = \o{e_2} = 2$. Moreover, $\ext_H(T,e_i)\leq 8$ for $i=1,2$, by Corollary~\ref{cor:compatible-edge-Kr}. 
Also, since $H[e_1\cup e_2]$ is triangle-free, $E_H(e_1,e_2)\leq 2$ and hence $\ext_H(e_1,e_2)\leq 4$.
Throughout the proof, we use each of these bounds unless the structure of $H$ allows us to obtain a better bound. 

First note that if there is at most one edge between $e_1$ and $e_2$, then $\ext_H(e_1,e_2)\leq 2$. In this case we use the bound
\begin{equation*}
\o{H} \leq \o{T}\cdot\o{e_1}\cdot\o{e_2}\cdot\ext_H(T,e_1)\cdot\ext_H(T,e_2)\cdot\ext_H(e_1,e_2),
\end{equation*}

\noindent to obtain $\o{H} \leq 6\cdot 2\cdot 2\cdot 8\cdot 8\cdot 2 < 2^{12}$, which allows us to restrict to graphs $H$ such that $H[e_1\cup e_2]\simeq K_{2,2}$.

We count the number of orientations by considering different values of $E_H(e_1\cup e_2,V(T))$. In particular, since $H$ is $K_4$-free, we have that $E_H(e_i, V(T))\leq 4$ for $i=1,2$. First note that if $E(e_i,V(T))=3$ for $i=1,2$, then $\ext_H(e_i,T)\leq 6$. Therefore, if there are at most six edges between $H[e_1\cup e_2]$ and $T$,
then either there are at most three edges between each \(e_i\) and \(T\),
which implies \(\ext_G(T,e_1), \ext_G(T,e_2)\leq 6\);
or, without loss of generality, there are at most two edges between \(e_1\) and \(T\),
which implies \(\ext_G(T,e_1)\leq 4\).
In both cases we have that $\ext_G(T,e_1)\cdot \ext_G(T,e_2)\leq 36$ and consequently that $\o{H}\leq 6\cdot 2 \cdot 2 \cdot 36 \cdot 4 < 2^{12}$.

Thus, we assume that $7\leq E_H(e_1\cup e_2, V(T))\leq 8$. 
Then, without loss of generality, we have $E_H(e_1,V(T))=4$ and, by Turán's Theorem, $H_1=H[e_1\cup V(T)]$ is isomorphic to $K_{1,2,2}$.
If $E_H(e_2,V(T))=3$,  the aforementioned bounds and  Lemma~\ref{prop:num_of_tfo-vertex-biclique} yields
\[\o{H}\leq \o{H_1}\cdot \ext_H(e_2,T)\cdot \ext(e_1,e_2)\leq 82\cdot 8\cdot 4 <2^{12}.\]
Finally, if $E_H(e_i,V(T))=4$ for both $i=1,2$, then the graphs $H_i=H[e_i\cup V(T)]$ are isomorphic to $K_{1,2,2}$ with $v_i\in V(T)$ being the vertex of degree $4$ in $H_i$. Since $H$ does not contain two disjoint triangles, then $v_1=v_2$ and since $H[e_1\cup e_2]\simeq K_{2,2}$, we have in fact $H\simeq K_{1,3,3}$. 
Finally, Lemma~\ref{prop:num_of_tfo-vertex-biclique} yields $\o{H}=2754<2^{12}$.
\end{proof}
 
In the remainder of this section we prove Theorem~\ref{thm:main}, which follows by induction on the number of vertices. Unfortunately, we need a slightly stronger base of induction than the one given by Proposition~\ref{prop:Alon-small-cases}, which is the content of the next proposition. We present its proof in the Appendix (Section~\ref{sec:appendix}).
Recall that the \emph{clique number} of a graph \(G\), denoted by \(\omega(G)\), 
is the size of a clique in \(G\)  with a maximum number of vertices.

\begin{proposition}\label{prop:small-cases}
	Let $G$ be an $n$-vertex graph.
	If $9\leq n\leq 7+\min\{\omega(G),8\}$, then $\o{G} \leq 2^\bound$.
	Furthermore, $\o{G} = 2^\bound$ if and only if $G \simeq K_{\lfloor n/2\rfloor,\lceil n/2\rceil}$.
\end{proposition}

We are now ready to prove Theorem~\ref{thm:main}, which is rewritten as follows:

\begin{theoremnull}[Theorem~\ref{thm:main}]\label{thm:large}
	Let $G$ be an $n$-vertex graph.
	If $n\geq 8$, then
	$
	\o{G} \leq 2^\bound.
	$	
	Furthermore, $\o{G} = 2^\bound$ if and only if $G \simeq K_{\lfloor n/2\rfloor,\lceil n/2\rceil}$.
\end{theoremnull}

\begin{proof}
Let $r=\omega(G)$.
The proof follows by induction on $n$.
By Proposition~\ref{prop:Alon-small-cases}, the statement holds for $n=8$. 
If $9\leq n\leq 10$, then the result follows from Mantel's Theorem (see~\cite{Ma1907}) for $r=2$, 
and from Proposition~\ref{prop:small-cases} for $r\geq 3$, as  $n \leq 7+\min\{r,8\}$.
Thus, assume $n\geq 11$ and suppose that the statement holds for any graph with less than $n$ vertices (but at least 8 vertices).

Let $K$ be a clique of $G$ of size $s=\min\{r,8\}$.
If $n\leq 7+s$, then the result follows from Proposition~\ref{prop:small-cases}, so we may assume that 
$n-s\geq 8$.
Thus, we can apply the induction hypothesis for any subgraph of $G$ with at least $n-s$ vertices.

If $r\geq 8$, then we have $s=8$.
By Proposition~\ref{prop:Alon-small-cases}, we have $\o{K}\leq 2^{16}$ and,
by Proposition~\ref{lemma:compatible-vertex-clique}, for each vertex $v\in V(G-K)$ we have $\ext_G(v,K) \leq 9$.
Therefore, applying the induction hypothesis to $G-K$ we have
\begin{align}
\label{eq:grandever1}
\o{G}	&\leq \o{K}\cdot \ext_G(G-K,K)\cdot \o{G-K}\nonumber\\
	&\leq   2^{16}\cdot 9^{n-8}\cdot 2^{(n-8)^2/4}
	< 2^\bound,
\end{align}
 where we used that $n-8\geq 1$.
From now on we assume that $r\leq 7$ and consequently that $s=r$.
Due to the different structure of the graphs with small clique numbers, we divide the rest of the proof according to the value of~$r$.

\smallskip\noindent
\textbf{Case $\mathbf{r\in\{5,6,7\}}$.}
Let $G'=G-K$.
Since $G$ is $K_{r+1}$-free, every vertex $v$ of $V(G')$ is adjacent to at most $r-1$ vertices of $K$.
Then, by Proposition~\ref{lemma:compatible-vertex-clique}, we have $\ext_G(v,K)\leq r$ for every $v\in V(G')$.
Therefore, the following holds for $r\in\{5,6,7\}$ and $n\geq 9$.
\begin{align}
\label{eq:grandever2}
	\o{G}	&\leq	 \o{K}\cdot \ext_G(G',K)\cdot \o{G'}\nonumber\\
		&\leq r!\cdot r^{n-r} \cdot \o{G'}\nonumber\\
		&\leq r! \cdot 2^{(n-r)\log_2 r}\cdot 2^{(n-r)^2/4}\nonumber\\
		&< 2^{\frac{r^2 + 2r(n-r)-1}{4} + \frac{(n-r)^2}{4}} \leq 2^\bound.
\end{align}

\smallskip\noindent
\textbf{Case $\mathbf{r=4}$.}
Let $G'=G-K$.
By the induction hypothesis, for any $u\in V(G)$, we have $\o{G-u}\leq 2^{\lfloor(n-1)^2/4\rfloor}$.
If $G$ contains a vertex $u$ with degree smaller than $(n-1)/2$,
then
\begin{equation}
\label{eq:grandever3}
 \o{G}< \o{G-u}\cdot 2^{d(u)}\leq 2^\bound.
 \end{equation}
Thus, we may assume that $\delta(G)\geq (n-1)/2$.
Since $n\geq 11$, we have $\delta(G)\geq 5$ and since $G$ is $K_5$-free, each vertex in $V(G')$
contains at most $3$ neighbors in $K$.
Hence, we have $\delta(G')\geq 2$.
Therefore, since $|V(G')|\geq 7$, there is a matching with at least two edges in~$G'$.

Let $y\geq 2$ be the size of a maximum matching $M$.
By Lemma~\ref{lemma:compatible-edge-Kr},
we have $\ext_G(e,K)\leq 13$, for every $e\in E(G')$.
Moreover, since every vertex in $V(G')$ has at most $3$ neighbors in $K$,
by Proposition~\ref{lemma:compatible-vertex-clique}, we have $\ext_G(v,K) \leq 4$
for every $v\in V(G')\setminus V(M)$.
Therefore, we have
\begin{align}
\label{eq:grandever4}
\o{G}	&\leq	 \o{K}\cdot \ext_G(V(M),K) \cdot \ext_G(V(G')\setminus V(M),K) \cdot \o{G'}\nonumber\\
	&\leq	4! \cdot 13^y \cdot 4^{n-4-2y}\cdot 2^{\lfloor (n-4)^2/4\rfloor}\nonumber\\
	&\leq 3\cdot \left(\dfrac{13}{16}\right)^y\cdot 2^3\cdot 2^{2(n-4)}\cdot 2^ {(n-4)^2/4} < 2^\bound,
\end{align}

\noindent as $3\cdot (13/16)^2 \leq 2^{3/4}$.

\smallskip\noindent
\textbf{Case $\mathbf{r=3}$.}
Let $\mathcal{T}$ be a maximum collection of vertex-disjoint triangles of $G$.
Set $G' = G - \cup_{T\in \mathcal{T}}V(T)$, let $M$ be a maximum matching in $G'$,
and let $Z = V(G')\setminus V(M)$.
Clearly, $G'$ is a $K_3$-free graph and $Z$ is an independent set.
Set $x= |\mathcal{T}|$, $y = |M|$ and $z = |Z|$ and note that $n = 3x + 2y + z$.

By Lemma~\ref{lemma:compatible-K3-K3}, we have $\ext_G(T_1,T_2)\leq 15$ for every $T_1,T_2\in\mathcal{T}$
and by Lemma~\ref{lemma:compatible-edge-Kr} we have $\ext_G(\{u,v\},T)\leq 8$ for every $\{u,v\}\in M$ and every $T\in\mathcal{T}$.
Moreover, since $G$ is $K_4$-free, by Proposition~\ref{lemma:compatible-vertex-clique}, we have $\ext_G(v,T)\leq 3$ for every $v\in Z$ and every $T\in\mathcal{T}$.
Since $G'$ is $K_3$-free, no vertex in $Z$ is adjacent to two vertices of the same edge in $M$,
and hence $\ext_G(u,\{v,w\})\leq 2$ for every $u\in Z$ and $\{u,v\}\in M$.
Finally, note that $\o{T}\leq 6$ for every $T\in\mathcal{T}$,
and since $G'$ is $K_3$-free, we have $\o{G[M]}\leq 2^{(2y)^2/4} = 2^{y^2}$.
Therefore, we have $\o{G} \leq	6^x \cdot 15^{x\choose 2}\cdot 8^{xy}\cdot 2^{y^2}\cdot 3^{xz}\cdot 2^{yz}=f(x,y,z)$.

\begin{claim}
\(f(x,y,z) < 2^\bound\) when (i) \(x\geq 3\) or (ii) \(z\geq 2\).
\end{claim}

\begin{claimproof}
Since $n=3x+2y+z$, we have that
\[ \frac{n^2-1}{4} = \frac{9x^2}{4}+3xy+y^2+\frac{3}{2}xz+yz+ -\frac{z^2-1}{4}. \]

\noindent We are left to prove that $x\log_26+\binom{x}{2}\log_2 15+xz\log_23\leq 9x^2/4+3xz/2-(z^2-1)/4$. By using the bounds $\log_2 15\leq 3.95$, $\log_2 6 \leq 2.6$ and $\log_23\leq1.6$ and multiplying the previous equation by $4$, we are left with the following inequality:
\begin{equation}
\label{eq:grandever5}
1.1x^2-2.5x-0.4xz+z^2-1>0.
\end{equation}

Note that $z^2-0.4xz\geq -0.04x^2$ and, moreover, that $x^2-2.5x-1>0$ for every $x\geq 3$. Finally, we are left with the case $x\in \{1,2\}$ and $z\geq 2$, which can be done by replacing each value of $x$ in~\eqref{eq:grandever5} and using that $z\geq 2$.
\end{claimproof}

\smallskip\noindent
Therefore, we may assume \(x\leq 2\) and \(z\leq 1\).
In this case, we need to explore the structure of the graph $G$ carefully.
Recall that  $y=|M|$ and $z=|Z|$, where $M$ is a maximum matching of $G'$ and $Z=V(G')\setminus V(M)$.
Since $n\geq 11$, we have $M\neq\emptyset$.

Suppose first that $x=2$ and let $T_1$ and $T_2$ be the triangles in $\mathcal{T}$.
Let $e$ be an edge of $M$ and $H=G[V(T_1)\cup V(T_2)\cup e \cup Z]$.
Since $|V(H)|\in\{8,9\}$ and $H$ is not a balanced complete bipartite graph, by Proposition~\ref{prop:small-cases}, we have $\o{H}< 2^{16+4z}$.
By Lemmas~\ref{lemma:compatible-vertex-clique} and~\ref{lemma:compatible-edge-Kr} and the fact that $G'$ is $K_3$-free, we have for every $e'\in M\setminus\{e\}$, that $\ext_G(e',Z)\leq 2^z$, $\ext_G(e',e)\leq 4$ and $\ext_G(e',T_i)\leq 8$ for $i=1,2$. We conclude that $\ext_G(e',H)\leq 2^z\cdot 4\cdot 8\cdot 8= 2^{8+z}$ for every $e'\in M\setminus\{e\}$.
Finally, we have $\ext_G(e',e'')\leq 4$ for every two edges $e'$ and $e''$ of $M$, and there are $2$ ways to orient each one of the $y-1$ edges of $M\setminus\{e\}$.
Therefore,
\begin{align}
		\label{eq:grandever6}
	\o{G}	&<	2^{16+4z} \cdot 2^{(8+z)(y-1)} \cdot 4^{y-1\choose 2} \cdot 2^{y-1}
		=	2^{((6+2y+z)^2-z)/4} = 2^{\lfloor n^2/4\rfloor},
\end{align}

\noindent where we used that $z^2=z$.

Thus, we may assume that $x=1$.
Let $T$ be the triangle in $\mathcal{T}$.
Since $n\geq 11$, we have $y\geq 2$.
Let  $e_1$ and $e_2$ be edges of $M$ and put $H=G[V(T)\cup e_1\cup e_2]$.
By Lemma~\ref{lemma:claim}, we have $\o{H}<2^{12}$.
For every $e\in M\setminus\{e_1,e_2\}$  we have
$\ext_G(e,e_1)\leq 4$ and $\ext_G(e,e_2)\leq 4$,
and, by Lemma~\ref{lemma:compatible-edge-Kr}, we have $\ext_G(e,T)\leq 8$,
and hence $\ext_G(e,H)\leq \ext_G(e,K)\cdot\ext_G(e,e_1)\cdot\ext_G(e,e_2) = 128$.
Also, by Proposition~\ref{lemma:compatible-vertex-clique},
for every vertex $u\notin V(T)\cup V(M)$ we have $\ext_G(u,T)\leq 3$,
and since $G'$ is $K_3$-free, $\ext_G(u,e) \leq 2$ for every $e\in M$.
Therefore, we have
\begin{equation}
\label{eq:grandever8}
	\o{G}	<	2^{12} \cdot 128^{y-2} \cdot 2^{(2y-4)^2/4}\cdot (3\cdot 2^y)^z
		\leq	2^{((3 + 2y+z)^2-1)/4}	< 2^\bound.
\end{equation}

\smallskip\noindent
\textbf{Case $\mathbf{r=2}$.}
Since $G$ is triangle-free, we have $\o{G} = 2^{|E(G)|}$.
Thus, by Mantel's Theorem, if $G$ is not isomorphic to $K_{\lfloor n/2\rfloor,\lceil n/2\rceil}$, we have
\begin{equation}
\label{eq:grandever9}
\o{G}< 2^\bound.
\end{equation}
Furthermore, $\o{G} = 2^\bound$ if and only if $G \simeq K_{\lfloor n/2\rfloor,\lceil n/2\rceil}$.
This completes the proof that for any \(n\)-vertex graph \(G\) with \(n\geq 8\), we have \(\o{G}\leq 2^\bound\).
Since inequalities~\eqref{eq:grandever1}--\eqref{eq:grandever9} are strict, we get that $\o{G} = 2^\bound$ if and only if $G \simeq K_{\lfloor n/2\rfloor,\lceil n/2\rceil}$, which concludes the proof of the theorem.
\end{proof}

\section{Open problems}
\label{sec:conc}

In this section we  discuss some open problems and directions for future research.
Given an oriented graph $\vec{H}$, recall that $D(n,\vec H)$ denotes the maximum number of $\vec{H}$-free orientations of~$G$, for all $n$-vertex graphs $G$.

\subsection{Avoiding an oriented graph}

In this paper we determine $D(n,\circ_3)$ for every possible $n$.
A natural problem is to extend our result to estimate the number of orientations of graphs avoiding strongly connected cycles $\circcycle_k$ for $k\geq 4$.
As far as we know, the following problem is open even for large $n$.

\begin{problem}\label{prob1}
Let $k\geq 4$. Determine $D(n,\circcycle_k)$ for every $n\geq 1$.
\end{problem}

An interesting problem is to determine $D(n,\vec H)$ for any oriented graph $\vec H$.
For a tournament $\vec T_k$ on $k$ vertices, $D(n,\vec T_k)$ was determined for sufficiently large $n$ by Alon and Yuster~\cite{AlYu06}.
For a moment, we consider edge colorings of graphs.
Denote by $F(n,k)$ the maximum number of $2$-edge colorings of a graph $G$ with no monochromatic $K_k$, 
among all graphs $G$ on \(n\) vertices.
The following result was proved by Yuster~\cite{Yu96} (for $k=3$) and Alon, Balogh, Keevash and Sudakov~\cite{AlBaKeSu04} (for $k\geq 4$).
\begin{lemma}\label{lemma:aux}
For every $k\geq 3$, there exists $n_0$ such that for all $n\geq n_0$ we have $F(n,k) = 2^\bound$.
\end{lemma}

Consider now the transitively oriented tournament $K_k^{\shortrightarrow}$ with $k$ vertices.
Using a simple argument, Alon and Yuster~\cite{AlYu06} used Lemma~\ref{lemma:aux} to prove that $D(n,K_3^{\shortrightarrow})=2^\bound$ for $n\geq 1$.
For $k\geq 4$, they proved that $D(n,K_k^{\shortrightarrow})=2^\bound$ for a (very) large $n$.
Thus, the following problem remains open.
\begin{problem}\label{prob1a}
Let $k\geq 4$. Determine $D(n,K_k^{\shortrightarrow})$ for every $n\geq 1$.
\end{problem}

\subsection{Avoiding families of oriented graphs}

Another direction of research arises when, instead of forbidding a fixed oriented graph, we forbid families of oriented graphs.
For example, one may consider orientations of graphs that avoid non-transitive tournaments.
Denote by $T_k(n)$ the maximum number of orientations of a graph $G$ in which \emph{every} copy of $K_k$ is transitively oriented, for every $n$-vertex graph $G$.
The following problem generalizes Theorem~\ref{thm:main}.
\begin{problem}\label{prob2}
Let $k\geq 4$. Determine $T_k(n)$ for every $n\geq 1$.
\end{problem}

Consider the number of orientations of graphs that avoids strongly connected tournaments.
We denote by $S_k(n)$ the maximum number of orientations of a graph $G$ in which \emph{no copy} of $K_k$ is strongly connected, for every $n$-vertex graph $G$.
 \begin{problem}\label{prob3}
Let $k\geq 4$. Determine $S_k(n)$ for every $n\geq 1$.
\end{problem}
Note that Problem~\ref{prob3} also generalizes Theorem~\ref{thm:main}.
We remark that it would be interesting to determine $T_k(n)$ and $S_k(n)$ even if only for very large $n$.
For related problems in the context of random graphs, the reader is referred to~\cites{AlKoMoPa14,CoKoMoMo20}.

\bibliographystyle{amsplain}

\bibliography{bibliografia}

\newpage
\section{Appendix}\label{sec:appendix}

Here we prove Proposition~\ref{prop:small-cases}, which states that for an $n$-vertex graph $G$ with $9\leq n\leq 7+\min\{\omega(G),8\}$ we have $\o{G} \leq 2^\bound$ and, furthermore, $\o{G} = 2^\bound$ if and only if $G \simeq K_{\lfloor n/2\rfloor,\lceil n/2\rceil}$.

Similarly to the proof of Theorem~\ref{thm:main},
we explore the structure of the graph $G$ depending on the size of its maximum clique.
By Mantel's Theorem, we have $\o{G}=2^\bound$ when $G$ is the balanced complete bipartite graph.
We show that if this is not the case, then
$
\o{G}<2^\bound.
$
To show that this holds we use straightforward computer methods to check some inequalities, 
namely, inequalities \eqref{eq:ver0}--\eqref{eq:ver8}.

\begin{proof}[Proof of Proposition~\ref{prop:small-cases}]
Let $G$ be an $n$-vertex graph and for simplicity put $r=\omega(G)$.
	Suppose $9\leq n\leq 7+\min\{r,8\}$ and let $W$ be a clique of size $|W|=\min\{r,8\}$ in $G$.
	Put $G' = G\setminus W$.
	Note that if $|W|=8$, then Proposition~\ref{prop:Alon-small-cases} implies $D(G',\circ_3)\leq (n-8)!$ and $D(G[W],\circ_3)\leq 2^{16}$, and Proposition~\ref{lemma:compatible-vertex-clique} implies $\ext_G(v,W)\leq 9$ for every $v\in V(G')$.
	Therefore, for every $9\leq n\leq 15=7+\min\{r,8\}$ we have
	\begin{equation}
	\label{eq:ver0}
	\o{G}\leq  (n-8)!\cdot 9^{n-8}\cdot2^{16} < 2^{\lfloor n^2/4\rfloor}.
	\end{equation}

	From now on we assume that $|W|\leq 7$, which implies $|W|=r$ and, from Proposition~\eqref{prop:Alon-small-cases}, we have $\o{G'} \leq (n-r)!$ and $D(G[W],\circ_3)\leq r!$.
Note that since $G$ has no clique of size $r+1$, for each $v\in V(G')$ we have $d_G(v,W)\leq r-1$, which implies from Proposition~\ref{lemma:compatible-vertex-clique} that $\ext_G(v,W)\leq r$ for every $v\in V(G')$.
Combining these facts, for $r\in\{6,7\}$ and $9\leq n\leq 7+r$ we have
	\begin{equation}
	\label{eq:ver1}
	\o{G}\leq (n-r)!\cdot r^{n-r}\cdot r!< 2^\bound.
	\end{equation}
	Therefore, we may assume that $r\leq 5$.
	Due to the different structure of the graphs with small clique numbers, we divide the rest of the proof according to the value of~$r$.

	\smallskip\noindent
	\textbf{Case $\mathbf{r=5}$.}
		Let $M$ be a maximum matching of $G'$, say with $x$ edges (\(0\leq x\leq \lfloor(n-5)/2\rfloor\)),
		and note that $G'' = G'[V(G') \setminus V(M)]$ is an independent set with $n-5-2x$ vertices.
		By Corollary~\ref{cor:compatible-edge-Kr}, we have $\ext_G(e,W) \leq 19$ for every $e\in M$.
		Therefore, for $9\leq n\leq 12$ and $2\leq x\leq \lfloor (n-5)/2\rfloor$, we have
		\begin{equation}
		\label{eq:ver2}
		\o{G} \leq (n-5)!\cdot 19^x\cdot 5^{n-5-2x}\cdot  5! < 2^\bound.
		\end{equation}
		Thus, we may assume that $x\leq 1$.
		This implies $G'$ is a star with at most $n-6$ edges or $G'$ is composed of one triangle and $n-8$ isolated vertices.
		Hence, $\o{G'}\leq 2^{n-6}$.
		Therefore, for $9\leq n\leq 12$ and $0\leq x\leq 1$, we have
		\begin{equation}
		\label{eq:ver3}
		\o{G} \leq 5!\cdot 19^x\cdot 5^{n-5-2x} \cdot 2^{n-6} < 2^\bound.
		\end{equation}

	\smallskip\noindent
	\textbf{Case $\mathbf{r=4}$.}
		First, suppose that $G'$ contains a clique $K$ with $4$ vertices.
		Let $G'' = G'[V(G')\setminus K]$ (note that $G''$ has $n-8$ vertices) and let $x$ be the number of edges in a maximum matching of $G''$ ($0\leq x\leq 1$).
		From Proposition~\ref{prop:Alon-small-cases} we have $\o{G[W\cup K]}< 2^{16}$ and from Proposition~\ref{lemma:compatible-vertex-clique}, since $G$ has no $K_5$, for every $v\in V(G'')$ we have $\ext_G(v,K)\leq 4$ and $\ext_G(v,W)\leq 4$.
		Furthermore, by Corollary~\ref{cor:compatible-edge-Kr}, for any edge $\{u,v\}$ of $G''$, we have $\ext_G(\{u,v\},K)\leq 13$ and $\ext_G(\{u,v\},W) \leq 13$.
		Therefore, for $9\leq n\leq 11$ and \(0\leq x\leq 1\) we have
		\begin{equation}
		\label{eq:ver4}
		\o{G}<  (n-8)!\cdot 13^{2x}\cdot 4^{2(n-8-2x)} \cdot 2^{16} <  2^\bound.
		\end{equation}
		
		Thus we may assume that $G'$ contains no copy of $K_4$.
		This allows us to use 
Lemma~\ref{lemma:compatible-K3-K3}.		
		Suppose that $G'$ contains two vertex-disjoint triangles.
		In this case, we have \(n\geq 10\).
		Let $V_1$ and $V_2$ be the vertex sets of these triangles, say \(V_2=\{u,v,w\}\), 
		and note that, since $n\leq 11$, there is one vertex
		that do not belong to $V_1\cup V_2$ in $G'$ if and only if \(n=11\).
		If \(n=11\), let \(z\) be this vertex.
		In this case, from Proposition~\ref{lemma:compatible-vertex-clique}, we have $\ext_G(z,V_1\cup V_2\cup W)\leq 3\cdot3\cdot 4 = 36$.
		Since $\ext_G(\{u,v\},W)\leq 13$ and $\ext_G(w,W) \leq 4$, we obtain that $\ext_G(V_2,W)\leq 52$.
		Note that $\o{G[W\cup V_1]}\leq 7!$ and  $\o{G[V_2]}\leq 6$ and, from Lemma~\ref{lemma:compatible-K3-K3} we obtain $\ext_G(V_1,V_2)\leq 15$. 
		Combining the above facts, we have
		\begin{align}
			\o{G}	&	\leq 6\cdot 15\cdot 52 \cdot 7! <  2^\bound, & &\text{for }n=10; \\
			\o{G}	&	\leq 6\cdot 36\cdot 15\cdot 52 \cdot 7! <  2^\bound,& &\text{for }n=11.\label{eq:ver5}
		\end{align}
				
		Thus, we may assume that $G'$ contains no two vertex-disjoint triangles.
		If $G'$ contains a triangle~$K$, then let $G'' = G'[V(G')\setminus K]$ (note that $G''$ has $n-7$ vertices) and let $x$ be the number of edges in a maximum matching of $G''$ ($0\leq x\leq \lfloor(n-7)/2)\rfloor$).
		Therefore, for $9\leq n\leq 11$ we have
		\begin{equation}
		\label{eq:ver6}
		\o{G}\leq  (2^x\cdot 4^{x\choose 2}) \cdot13^x \cdot 8^x\cdot 3^{n-7-2x} \cdot 4^{n-7-2x}\cdot 2^{x(n-7-2x)}\cdot 7! <  2^\bound.
		\end{equation}
		
		Finally, assume that $G'$ contains no triangles.
		Then, similarly as before, letting $x$ be the number of edges in a maximum matching of $G'$
		 ($0\leq x\leq \lfloor(n-4)/2)\rfloor$), for $9\leq n\leq 11$ we have
		\begin{equation}
		\label{eq:ver7}
		\o{G}\leq  (2^x\cdot 4^{x\choose 2}) \cdot13^x  \cdot 4^{n-4-2x}\cdot 2^{x(n-4-2x)}\cdot 4! <  2^\bound.
		\end{equation}

	\smallskip\noindent
	\textbf{Case $\mathbf{r=3}$.}
		In this case the graph $G$ has $9\leq n\leq 10$ vertices.
		We start by noticing that if $G$ contains three vertex-disjoint triangles, 
		then there are six possible orientations of the edges of each triangle and, by Lemma~\ref{lemma:compatible-K3-K3}, there are at most fifteen ways to orient the edges between the triangles. 
		Let $y$ be the number of vertices that are not in these triangles.
		Note that \(0\leq y\leq 1\) and \(y=1\) if and only if \(n=10\).
Since $G$ is $K_4$-free, in case $y=1$, Proposition~\ref{lemma:compatible-vertex-clique} implies that there are $3$ ways to orient the edges between the vertex outside the triangles and each of the triangles.
Therefore, for $9\leq n\leq 10$ we have
		\begin{equation}
		\label{eq:ver8}
		\o{G}\leq 6^3\cdot 15^3\cdot 3^{3y} <  2^\bound.
		\end{equation}

		From the above discussion, we may assume that $G$ contains at most two vertex-disjoint triangles.
		For the rest of the proof we have to analyze the structure of $G$ carefully.
		Thus we consider two possible cases, depending on the number of vertices of $G$.

		\smallskip\noindent
		\textbf{Subcase $\mathbf{n=9}$.}
		First suppose that $\delta(G)\leq 4$.
		Let $u$ be a vertex of minimum degree and note that if $u$ is contained in a triangle, then $\ext_G(u,G-u) \leq 3\cdot 2^2 < 2^4$ and by Proposition~\ref{prop:Alon-small-cases}, we have $\o{G-u}\leq 2^{16}$.
		In case no triangle contains $u$, Proposition~\ref{prop:Alon-small-cases} gives $\o{G-u}< 2^{16}$ and $\ext_G(u,G-u) \leq 2^4$.
		Therefore, we obtain
		\begin{equation}
		\label{eq:ver9}
		\o{G}\leq \o{G-u}\cdot \ext_G(u,G-u)< 2^{20} = 2^\bound.
		\end{equation}
		
		Thus we may assume \(\delta(G)\geq 5\).
		Suppose that $G$ contains two vertex-disjoint triangles with vertex sets $V_1$ and $V_2$ (recall that $G$ contains at most two vertex-disjoint triangles).
		Let $G'$ be the subgraph of $G$ induced by the vertices that are not in $V_1$ or $V_2$.
		Thus, since $G$ is $K_4$-free, each vertex of $G'$ has at most two neighbors
		in each of $V_1$ and $V_2$.
		Since $\delta(G)\geq 5$ and $G'$ is triangle-free,
		$G'$ is an induced path of length $2$, say $uvw$.
		Moreover, each of the vertices $u$ and $w$ has two neighbors in $V_1$ and also in $V_2$.
		The vertex $v$ has two neighbors in one of the triangles, say in the set $V_1$.
		Since $G$ is $K_4$-free, $u$ and $v$ have only one common neighbor in $V_1$, which implies that the subgraph \(H\) of $G$ induced by the vertices $V_1\cup\{u,v\}$ is a $K_{1,2,2}$.	
		Thus, by Proposition~\ref{prop:num_of_tfo-vertex-biclique}, we have $\o{H}\leq 82$.
		Also, $H'=G[V_2\cup\{w\}]$ is a copy of $K_4^-$, and hence $\o{H'}\leq 6\cdot 3=18$.
		Finally, applying Lemmas~\ref{lemma:compatible-vertex-clique},~\ref{lemma:compatible-vertex-K4-} and~\ref{lemma:compatible-K3-K3}, we obtain
		$\ext_G(u,V_2)\leq 3$,
		$\ext_G(w,V_1)\leq 3$,
		$\ext_G(V_1,V_2)\leq 15$,
		$\ext_G(v,V(H'))\leq 5$,
		and hence
		\begin{equation}
		\label{eq:ver10}
		\o{G}\leq 82\cdot 18\cdot 15\cdot 3\cdot 3\cdot 5 < 2^{20} = 2^\bound.
		\end{equation}

		Assume that $G$ contains one triangle with vertex set $V_1=\{u_1,u_2,u_3\}$,
		but does not contain two vertex-disjoint triangles.
		Let \(G'\) be the subgraph of \(G\) induced by the vertices that are not in \(V_1\).
		Since \(\delta(G)\geq 5\) and no vertex in \(G'\) is adjacent to more than two vertices in \(V_1\), 
		we have \(\delta(G')\geq 3\).
		Thus, by Mantel's Theorem \(G'\) is isomorphic to \(K_{3,3}\).
		It is not hard to show that, since  $G$ is $K_4$-free and $\delta(G)\geq 5$, the graph $G$ is isomorphic to the graph $K_{1,4,4}$.
		Therefore, by Proposition~\ref{prop:num_of_tfo-vertex-biclique}, we have
		\begin{equation}
		\label{eq:ver10}
		\o{G}=271614< 2^{20}=2^\bound.
		\end{equation}

		\smallskip\noindent
		\textbf{Subcase $\mathbf{n=10}$.}
		We proceed similarly as in the case above.
		First suppose that $\delta(G)\leq 5$ and let $u$ be a vertex of minimum degree.
		If $u$ is contained in a triangle, then the previous subcase for graphs with $9$ vertices gives $\ext_G(u,G-u) \leq 3\cdot 2^3 < 2^5$ and also by the previous case (or Mantel's Theorem in case $G-u$ is $K_3$-free) we have $\o{G-u}\leq 2^{20}$.
		On the other hand, if there is no triangle that contains $u$, 
		then the previous subcase for graphs with $9$ vertices gives $\o{G-u}< 2^{20}$ because \(G-u\) contains a triangle, and hence we have $\ext_G(u,G-u) \leq 2^5$.
		Therefore, we have
		\begin{equation}
		\label{eq:ver11}
		\o{G}\leq \o{G-u}\cdot \ext_G(u,G-u)< 2^{25} = 2^\bound.
		\end{equation}
		
		Thus, we may assume that $\delta(G)\geq 6$.
		Suppose that $G$ contains two vertex-disjoint triangles with vertex sets $V_1$ and $V_2$ (recall that $G$ contains at most two vertex-disjoint triangles).
		Let $G'$ be the subgraph of $G$ induced by the vertices that are not in $V_1$ or $V_2$.
		Note that since $G$ is $K_4$-free and $G'$ is triangle-free, the graph $G'$ is a cycle and each vertex of $G'$ has exactly two neighbors in each of $V_1$ and $V_2$.
		Let $a_1b_1$ and $a_2b_2$ be two non-incident edges of $G'$
		and put $H_i = G[V_i\cup \{a_i,b_i\}]$, for $i\in\{1,2\}$.
		Note that $H_1$ and $H_2$ are isomorphic to $K_{1,2,2}$.
		Then, by Proposition~\ref{prop:num_of_tfo-vertex-biclique}, we have $\o{H_1}, \o{H_2}\leq 82$.
		Analogous to the subcase for graphs with $9$ vertices, we have $\ext_G(a_ib_i,V_{3-i}) \leq 8$ for $1\leq i\leq 2$, $\ext_G(a_1b_1,a_2b_2)\leq 4$, and $\ext_G(T_1,T_2)\leq 15$.
		Therefore, we have
		\begin{equation}
		\label{eq:ver12}
		\o{G}\leq  82^2\cdot 15\cdot 8^2\cdot 4 < 2^{25} = 2^\bound.	
		\end{equation}
		
		Assume that $G$ contains one triangle with vertex-set $V_1$, but does not contain two vertex-disjoint triangles.
		Let $G' = G - V_1$ and note that $G'$ is a triangle-free graph with~$7$ vertices.
		By Mantel's Theorem, $|E(G')|\leq 12$.
		Since $\delta(G)\geq 6$, and every vertex of $G'$ has at most two neighbors in $V_1$,
		we have $\delta(G')\geq 4$, which implies that $|E(G')|\geq 14$, a contradiction.

		\smallskip\noindent
		\textbf{Case  $\mathbf{r=2}$.}
		Since $G$ is triangle-free, we have $\o{G} = 2^{|E(G)|}$.
		Thus, by Mantel's Theorem, if $G$ is not isomorphic to $K_{\lfloor n/2\rfloor,\lceil n/2\rceil}$, we have
		\begin{equation}
		\label{eq:ver13}
		\o{G}< 2^\bound.
		\end{equation}
		Furthermore, $\o{G} = 2^\bound$ if and only if $G \simeq K_{\lfloor n/2\rfloor,\lceil n/2\rceil}$, which completes the proof that for any $n$-vertex graph $G$ with $9\leq n\leq 7+\min\{\omega(G),8\}$ we have $\o{G} \leq 2^\bound$.
		Since inequalities~\eqref{eq:ver0}--\eqref{eq:ver13} are strict, we get that $\o{G} = 2^\bound$ if and only if $G \simeq K_{\lfloor n/2\rfloor,\lceil n/2\rceil}$, which concludes the proof of the proposition.
\end{proof}

\end{document}